\documentclass[a4paper,pdftex]{amsart}
\usepackage{amsmath,amssymb,amsbsy,amsthm,amsfonts,comment,tikz}
\usepackage[all]{xy}
\usepackage{graphicx,color}
\usepackage{hyperref}
\hypersetup{
    colorlinks=true,
    linkcolor=purple,
    urlcolor=orange,
    citecolor=teal,
    pdfauthor = {Ryuma Orita},
    pdftitle = {Maximal topological complexity of monotone symplectic 4-manifolds},
    pdfpagemode = UseNone,
    bookmarksnumbered=true,
}

\title[Topological complexity of symplectic 4-manifolds]{Maximal topological complexity of monotone symplectic 4-manifolds}

\author{Ryuma Orita}
\address{Department of Mathematics, Faculty of Science, Niigata University, Niigata 950-2181, Japan}
\email{\href{mailto:orita@math.sc.niigata-u.ac.jp}{orita@math.sc.niigata-u.ac.jp}}
\urladdr{\url{https://ryuma-orita.netlify.app/}}

\subjclass[2020]{Primary 55M30; Secondary 53D05}
\keywords{Topological complexity, Lusternik--Schnirelmann category, symplectic manifolds, Kodaira dimension, ruled surfaces}
\thanks{This work was supported by JSPS KAKENHI Grant Numbers 21K13787 and 26K16979.}

\newtheorem{theorem}{Theorem}[section]
\newtheorem{lemma}[theorem]{Lemma}
\newtheorem{proposition}[theorem]{Proposition}
\newtheorem{corollary}[theorem]{Corollary}

\theoremstyle{definition}
\newtheorem{definition}[theorem]{Definition}

\newtheorem{problem}[theorem]{Problem}

\theoremstyle{remark}
\newtheorem{remark}[theorem]{Remark}

\newcommand{\ZZ}{\mathbb{Z}}
\newcommand{\QQ}{\mathbb{Q}}
\newcommand{\RR}{\mathbb{R}}

\newcommand{\Hom}{\operatorname{Hom}}
\newcommand{\Image}{\operatorname{Im}}
\newcommand{\Ker}{\operatorname{Ker}}

\newcommand{\relmiddle}[1]{\mathrel{}\middle#1\mathrel{}}

\newcommand{\pr}{\mathrm{pr}}
\newcommand{\cat}{\mathsf{cat}}
\newcommand{\TC}{\mathsf{TC}}

\newcommand{\wgt}{\mathrm{wgt}}

\newcommand{\CP}{\mathbb{C}P}
\newcommand{\barCP}{\overline{\mathbb{C}P}^2}

\begin{document}

\begin{abstract}
We continue the study of Farber's topological complexity for monotone symplectic manifolds initiated in \cite{Or25}.
First, we show that a closed spherically monotone symplectic manifold whose fundamental group contains no subgroup isomorphic to $\ZZ\oplus\ZZ$ is automatically toroidally monotone, with the same monotonicity constant.
As a consequence, every closed $4$-dimensional spherically monotone symplectic manifold whose Kodaira dimension is not $-\infty$ and whose fundamental group contains no $\ZZ\oplus\ZZ$ (for instance, is Gromov hyperbolic) has maximal topological complexity $\TC(M)=9$.
This settles, under strictly weaker hypotheses, the dichotomy $\TC(M)\in\{8,9\}$ left open there.
Second, we compute the topological complexity and the Lusternik--Schnirelmann category of all blowups of $S^2$-bundles over closed orientable surfaces of genus $g\geq 2$: they satisfy $\cat(M)=4$ and $\TC(M)=7$.
In particular, the hypothesis on the Kodaira dimension in the first result cannot be removed,
and closed symplectic $4$-manifolds realize the pairs $(\cat(M),\TC(M))=(3,5)$, $(4,7)$, $(5,9)$ in the three regimes considered in this paper.
Throughout, $\TC$ and $\cat$ are taken in the unreduced convention.
\end{abstract}

\maketitle

\tableofcontents


\section{Introduction and main results}\label{sec:introduction}

Farber's \textit{topological complexity} $\TC(X)$ \cite{Fa03,Fa04} is a numerical homotopy invariant of a path-connected topological space $X$
which measures the discontinuity of any motion planning algorithm on the configuration space $X$ of a mechanical system.
Writing $X^I$ for the space of paths $\gamma\colon I=[0,1]\to X$ and
\[
	p\colon X^I\to X\times X,\qquad p(\gamma)=(\gamma(0),\gamma(1)),
\]
for the free path fibration, $\TC(X)$ is the minimal number $k$ of open sets $U_1,\ldots,U_k$ covering $X\times X$
in such a way that $p$ admits a continuous section over each $U_i$;
if no finite cover of this kind exists, one sets $\TC(X)=\infty$.
As in \cite{Or25}, we use the unreduced convention;
part of the literature works with the reduced version, which is smaller by one.
If $X$ is a connected CW-complex, then
\begin{equation}\label{eq:upper_bound}
    \TC(X)\leq 2\dim{X}+1,
\end{equation}
see \cite[Theorem 4]{Fa03},
and $\TC(X)$ is called \textit{maximal} \cite{CV21} if the equality holds in \eqref{eq:upper_bound}.
A closely related invariant is the \textit{Lusternik--Schnirelmann category} $\cat(X)$,
i.e., the minimal number of open sets, each contractible in $X$, needed to cover $X$ (again in the unreduced convention);
one has \cite[Theorem 5 and its proof]{Fa03}
\begin{equation}\label{eq:TC_cat}
    \cat(X)\leq\TC(X)\leq\cat(X\times X)\leq 2\dim X+1.
\end{equation}
We refer the reader to \cite{Fa06,Fa08} and to the introduction of \cite{Or25} for background.

\textit{
    Following \cite{Or25}, all symplectic manifolds in this paper are assumed to be connected and not simply connected, unless explicitly stated otherwise.
}

\subsection{From spherical to toroidal monotonicity}

In the predecessor \cite{Or25} of the present paper, we studied $\TC$ for \textit{monotone} symplectic manifolds
and proved the following.
Here $\kappa(M,\omega)$ denotes the Kodaira dimension of a symplectic $4$-manifold,
and we refer the reader to Section \ref{sec:preliminaries} for the definitions of the toroidal and spherical monotonicity.

\begin{theorem}[{\cite[Theorem 1.4]{Or25}}]\label{thm:previous}
    Let $(M,\omega)$ be a closed $4$-dimensional symplectic manifold whose Kodaira dimension is not $-\infty$.
    \begin{enumerate}
        \item If $(M,\omega)$ is toroidally monotone, then $\TC(M)=9$.
        \item If $(M,\omega)$ is spherically monotone and the fundamental group $\pi_1(M)$ is of type $FL$
        and the centralizer of every non-trivial element of $\pi_1(M)$ is infinite cyclic,\footnote{In \cite{Or25} this hypothesis is misstated as ``every non-trivial element of the center is infinitely cyclic''. The hypothesis intended, and the one used in the proof via \cite[Theorem 3]{FM20}, is the one stated here; the same correction applies to \cite[Theorem 4.4]{Or25}.}
        then $\TC(M)$ is either $8$ or $9$.
    \end{enumerate}
\end{theorem}

The dichotomy in Theorem \ref{thm:previous} (ii) has its origin in the theory of topological complexity of \textit{aspherical} spaces, which we now recall.
By the classical theorem of Eilenberg and Ganea \cite{EG57},
the Lusternik--Schnirelmann category of an aspherical space is determined by its fundamental group alone:
$\cat(K(\pi,1))=\mathrm{cd}(\pi)+1$, where $\mathrm{cd}$ denotes the cohomological dimension.
By the homotopy invariance of $\TC$,
the topological complexity of $K(\pi,1)$ also depends only on $\pi$,
and one writes $\TC(\pi):=\TC(K(\pi,1))$;
however, no algebraic description of $\TC(\pi)$ analogous to the Eilenberg--Ganea theorem is known,
and the problem of computing $\TC(\pi)$, posed by Farber \cite{Fa06}, remains open in general.
Among the many partial results
we mention the computations for orientable surfaces and tori \cite{Fa03},
for non-orientable surfaces \cite{Dr16,CV17},
the characterization of the groups with $\TC(\pi)=2$ \cite{GLO13},
upper bounds for nilpotent groups \cite{Gr12},
lower bounds in terms of pairs of subgroups whose conjugates intersect trivially \cite{GLO15b},
a mapping theorem for $\TC$ \cite{GLO15a},
and the description of $\TC(\pi)$ in terms of Bredon cohomology
by Farber, Grant, Lupton, and Oprea \cite{FGLO19};
see also \cite{CF10,CV21}.
A milestone was reached by Farber and Mescher \cite{FM20}.

\begin{theorem}[{\cite[Theorem 1]{FM20}}]\label{thm:FM}
    Let $X$ be a connected aspherical finite cell complex whose fundamental group $\pi=\pi_1(X)$ is hyperbolic in the sense of Gromov.
    Then $\TC(X)$ equals either $\mathrm{cd}(\pi\times\pi)$ or $\mathrm{cd}(\pi\times\pi)+1$.
\end{theorem}

Here $\mathrm{cd}(\pi\times\pi)=2\,\mathrm{cd}(\pi)$ for every hyperbolic group $\pi$
by a theorem of Dranishnikov (see \cite{Dr20} and the references therein),
and the general upper bound $\TC(X)\leq\mathrm{cd}(\pi\times\pi)+1$
follows from \eqref{eq:TC_cat} applied to a CW model of $X\times X$ of minimal dimension
together with the Eilenberg--Ganea theorem (see \cite[Section 2]{FM20}).\footnote{The
Eilenberg--Ganea theorem is used here through the equality of $\mathrm{cd}(\pi\times\pi)$
with the minimal dimension of a CW model of $K(\pi\times\pi,1)$.
This equality is classical for $\mathrm{cd}\neq 2$,
the case $\mathrm{cd}=2$ being the open Eilenberg--Ganea conjecture;
it also holds in the only remaining case relevant here,
since $\mathrm{cd}(\pi\times\pi)=2$ forces $\mathrm{cd}(\pi)\leq 1$ (multiply a top-dimensional class of the first factor by the Berstein--Schwarz class of the second, so that
$\mathrm{cd}(\pi\times\pi) \geq \mathrm{cd}(\pi)+1$
for any non-trivial $\pi$),
hence $\pi$ is free by the Stallings--Swan theorem (see \cite[Chapter VIII]{Br82}),
and then $K(\pi,1)\times K(\pi,1)$ admits a $2$-dimensional model.}
Thus Theorem \ref{thm:FM} states that $\TC(X)$ equals either $2\,\mathrm{cd}(\pi)$ or $2\,\mathrm{cd}(\pi)+1$,
i.e., is either maximal or one less than maximal.
The proof of \cite{FM20} proceeds by an obstruction theory for so-called essential cohomology classes,
built out of the powers of the canonical class $v\in H^1(\pi\times\pi;I)$;
the loss of one in Theorem \ref{thm:FM} can be traced to a single surviving obstruction,
governed by the first cohomology of the centralizers of elements of $\pi$
(see Remark \ref{rem:comparison} below),
and Farber and Mescher asked whether the relevant top power of the canonical class is always non-zero,
a positive answer to which would yield maximality.
The dichotomy was indeed resolved, by Dranishnikov, in the maximal direction.

\begin{theorem}[{\cite[Theorem 3.0.2]{Dr20}}]\label{thm:Dr}
    Let $\pi$ be a finitely generated torsion-free hyperbolic group with $\mathrm{cd}(\pi)=n\geq 2$.
    Then $\TC(K(\pi,1))=2n+1$; that is, $\TC(K(\pi,1))$ is maximal.
\end{theorem}

In \cite{Dr20} the normalized convention is used, in which the statement reads $\TC(\pi)=2n$;
by the equality $\mathrm{cd}(\pi\times\pi)=2\,\mathrm{cd}(\pi)$ recalled above,
this is the maximal option of Theorem \ref{thm:FM}.
Dranishnikov's proof is based on the Bredon cohomology characterization of $\TC(\pi)$ established in \cite{FGLO19};
we refer the reader to \cite{HL22} for a generalization to higher topological complexity and to toral relatively hyperbolic groups.

The predecessor \cite{Or25} transplanted the Farber--Mescher stage of this development
from aspherical spaces to (typically non-aspherical) spherically monotone symplectic $4$-manifolds:
Theorem \ref{thm:previous} (ii) was proved by pulling the aspherical class $[\omega]-\lambda c_1$
back from the Eilenberg--MacLane space $K(\pi_1(M),1)$ via \cite[Lemma 2.1]{RT99}
and applying the weight estimate of \cite[Theorem 3]{FM20},
so that the dichotomy of Theorem \ref{thm:FM} was inherited.
Our first main result performs the Dranishnikov stage in this symplectic setting.
Rather than upgrading the algebraic machinery of \cite{FM20,Dr20}, however, we convert the geometry:
we show that, under a mild group-theoretic hypothesis,
spherical monotonicity automatically improves to toroidal monotonicity,
whereupon Theorem \ref{thm:previous} (i) ---
that is, ultimately, the atoroidal weight estimate of Grant and Mescher \cite{GM20} ---
applies and yields the maximal value directly.

\begin{theorem}\label{thm:A}
    Let $(M,\omega)$ be a $2n$-dimensional closed symplectic manifold
    which is spherically monotone or spherically negative monotone with monotonicity constant $\lambda$.
    Assume that
    \begin{enumerate}
        \item $\pi_1(M)$ contains no subgroup isomorphic to $\ZZ\oplus\ZZ$, and
        \item $([\omega]-\lambda c_1)^n\neq 0\in H^{2n}(M;\RR)$.
    \end{enumerate}
    Then $(M,\omega)$ is toroidally monotone or toroidally negative monotone with the same monotonicity constant $\lambda$.
    In particular,
    \[
        \TC(M)=4n+1
        \quad\text{and}\quad
        \cat(M)=2n+1.
    \]
\end{theorem}

\begin{corollary}\label{cor:A_dim4}
    Let $(M,\omega)$ be a closed $4$-dimensional spherically monotone symplectic manifold whose Kodaira dimension is not $-\infty$.
    If $\pi_1(M)$ contains no subgroup isomorphic to $\ZZ\oplus\ZZ$
    --- for instance, if $\pi_1(M)$ is a Gromov hyperbolic group ---
    then
    \[
        \TC(M)=9
        \quad\text{and}\quad
        \cat(M)=5.
    \]
\end{corollary}

The key to Theorem \ref{thm:A} is the elementary but remarkably effective observation,
essentially due to Borat \cite{Bo16} (see also \cite[Proposition 6.2]{SS25}),
that an aspherical degree two cohomology class of a space whose fundamental group contains no $\ZZ\oplus\ZZ$ is automatically atoroidal;
see Lemma \ref{lem:aspherical_implies_atoroidal}.
This converts spherical monotonicity into toroidal monotonicity,
and then \cite[Theorem 4.3]{Or25} applies.
In particular, the proof of Theorem \ref{thm:A} does not pass through the Eilenberg--MacLane space $K(\pi_1(M),1)$
and uses neither \cite[Lemma 2.1]{RT99} nor the results of \cite{FM20};
see Remark \ref{rem:comparison} for a detailed comparison with Theorem \ref{thm:previous} (ii).

Combining Corollary \ref{cor:A_dim4} with the classification of symplectic $4$-manifolds with Kodaira dimension $-\infty$
due to Li--Liu \cite{LL95}, Liu \cite{Liu96}, and Ohta--Ono \cite{OO96a,OO96b},
the hypothesis on the Kodaira dimension can be traded for a purely group-theoretic one.

\begin{corollary}\label{cor:group_theoretic}
    Let $(M,\omega)$ be a closed $4$-dimensional spherically monotone symplectic manifold
    such that $\pi_1(M)$ contains no subgroup isomorphic to $\ZZ\oplus\ZZ$
    and is not isomorphic to $\pi_1(\Sigma_g)$ for any $g\geq 2$.
    Then $\TC(M)=9$ and $\cat(M)=5$.
\end{corollary}

We prove Corollary \ref{cor:group_theoretic} in Section \ref{sec:spherical_to_toroidal}.

Let us close this subsection with a discussion of the scope of Theorem \ref{thm:A}
and of the problem of finding examples.
First, the case $\lambda=0$ of Theorem \ref{thm:A} is precisely the case of an aspherical symplectic form.
In this case hypothesis (ii) holds automatically, since $([\omega]-\lambda c_1)^n=[\omega]^n$ is a volume class,
and Theorem \ref{thm:A} recovers the theorem of Grant and Mescher \cite{GM20} combined with Borat's theorem \cite{Bo16}:
a closed symplectic manifold with aspherical $\omega$ whose fundamental group contains no $\ZZ\oplus\ZZ$
has maximal topological complexity.
The new content of Theorem \ref{thm:A} is therefore the genuinely monotone case $\lambda\neq 0$,
where the class $[\omega]-\lambda c_1$, rather than $[\omega]$ itself, is aspherical.

Second, as in \cite{Or25}, the question of exhibiting specific examples in dimension four remains delicate.
By the theorem of Gompf \cite[Theorem 0.1]{Go95},
every finitely presentable group --- in particular, every hyperbolic group ---
can be realized as the fundamental group of a closed symplectic $4$-manifold;
moreover, hyperbolic groups are generic among finitely presented groups
in the models of random groups, as emphasized in \cite{FM20}.
Corollary \ref{cor:group_theoretic} therefore predicts maximal topological complexity
for a potentially vast class of symplectic $4$-manifolds.
However, we do not know how to arrange spherical monotonicity with $\lambda>0$ in Gompf's construction.
We also note that the product construction of \cite[Remark 1.7]{Or25} does not produce such examples:
for a product $M_1\times N$ of a symplectically atoroidal (or aspherical) manifold $M_1$
with a strongly monotone manifold $N$,
the class $[\omega]-\lambda c_1$ is pulled back from $M_1$,
so that its top power vanishes and hypothesis (ii) of Theorem \ref{thm:A} fails ---
see Remark \ref{rem:sharpness}, where this failure is shown, in an instance,
to be essential rather than an artifact of the method.
We therefore pose the following problem.

\begin{problem}\label{prob:examples}
    Construct a closed spherically monotone symplectic $4$-manifold with monotonicity constant $\lambda>0$
    whose fundamental group contains no subgroup isomorphic to $\ZZ\oplus\ZZ$
    and is not isomorphic to the fundamental group of a closed orientable surface.
\end{problem}

By Corollary \ref{cor:group_theoretic}, any manifold as in Problem \ref{prob:examples} satisfies $\TC(M)=9$.

\subsection{Ruled surfaces and their blowups}

Our second theme is the case $\kappa(M,\omega)=-\infty$, which is excluded from Theorem \ref{thm:previous} and Corollary \ref{cor:A_dim4}.
By the classification results of Li--Liu \cite{LL95}, Liu \cite{Liu96}, and Ohta--Ono \cite{OO96a,OO96b},
the closed symplectic $4$-manifolds of Kodaira dimension $-\infty$ are precisely
the blowups of rational and ruled symplectic $4$-manifolds.
The simply connected ones are covered by the theorem of Farber--Tabachnikov--Yuzvinsky \cite{FTY03}:
they satisfy $\TC(M)=\dim M+1=5$.
The remaining ones are the blowups of $S^2$-bundles over a closed orientable surface $\Sigma_g$ of genus $g\geq 1$.
Our second main result computes their $\TC$ and $\cat$ completely for $g\geq 2$.

\begin{theorem}\label{thm:B}
    Let $g\geq 2$ and $k\geq 0$, and let $M$ be a smooth closed $4$-manifold obtained from the total space of an $S^2$-bundle over $\Sigma_g$
    by blowing up $k$ points, i.e., $M\cong E\mathbin{\#}k\barCP$ where $E\to\Sigma_g$ is an $S^2$-bundle.
    Then
    \[
        \cat(M)=4
        \quad\text{and}\quad
        \TC(M)=7.
    \]
\end{theorem}

Combining Theorem \ref{thm:B} with Theorem \ref{thm:simply_connected_ruled} below and Corollary \ref{cor:A_dim4},
we obtain the following trichotomy;
note the arithmetic progression of the pairs $(\cat,\TC)$.

\begin{corollary}\label{cor:trichotomy}
    Let $(M,\omega)$ be a closed symplectic $4$-manifold.
    \begin{enumerate}
        \item If $M$ is simply connected with $\kappa(M,\omega)=-\infty$, then $(\cat(M),\TC(M))=(3,5)$.
        \item If $\kappa(M,\omega)=-\infty$ and $M$ is a blowup of an $S^2$-bundle over $\Sigma_g$ with $g\geq 2$, then $(\cat(M),\TC(M))=(4,7)$.
        \item If $\kappa(M,\omega)\neq-\infty$, $(M,\omega)$ is spherically monotone, and $\pi_1(M)$ contains no $\ZZ\oplus\ZZ$, then $(\cat(M),\TC(M))=(5,9)$.
    \end{enumerate}
\end{corollary}

The genus one case is more subtle and we obtain only partial results;
see Proposition \ref{prop:genus_one} and Problem \ref{prob:genus_one}.

Theorem \ref{thm:B} also shows that the hypotheses of Theorem \ref{thm:A} and Corollary \ref{cor:A_dim4} are indispensable:
for $g\geq 2$ the product $\Sigma_g\times S^2$ is \textit{toroidally} monotone with hyperbolic fundamental group,
yet $\TC(\Sigma_g\times S^2)=7\neq 9$ and $\cat(\Sigma_g\times S^2)=4\neq 5$;
see Remark \ref{rem:sharpness} for details.

\subsection*{Organization of the paper}
In Section \ref{sec:preliminaries} we recall the necessary definitions from \cite{Or25}.
In Section \ref{sec:spherical_to_toroidal} we prove Theorem \ref{thm:A} and Corollaries \ref{cor:A_dim4} and \ref{cor:group_theoretic}.
In Section \ref{sec:ruled} we prove Theorem \ref{thm:B} and its complements.


\section{Preliminaries}\label{sec:preliminaries}

In this section we fix our conventions, which are those of \cite[Section 2]{Or25}.

For a connected CW-complex $X$, let $\mathcal{L}X$ be its free loop space and,
given a free homotopy class $\alpha\in [S^1,X]$,
let $\mathcal{L}_{\alpha}X\subset\mathcal{L}X$ be the connected component consisting of the loops in the class $\alpha$.
Two homomorphisms into $H_2(X;\ZZ)$ play a central role in what follows.
The first is the Hurewicz homomorphism
\[
	h_S\colon\pi_2(X)\to H_2(X;\ZZ),\qquad [w]\mapsto w_*([S^2]),
\]
where $w\colon S^2\to X$ is a map representing a given element of $\pi_2(X)$.
The second is defined on the fundamental group of $\mathcal{L}_{\alpha}X$:
under the exponential law, an element of $\pi_1(\mathcal{L}_{\alpha}X)$ corresponds to a map $w\colon S^1\times S^1\to X$,
and we set
\[
	h_T\colon\pi_1(\mathcal{L}_{\alpha}X)\to H_2(X;\ZZ),\qquad [w]\mapsto w_*([S^1\times S^1]).
\]
For the reader's convenience, we reproduce the following two definitions from \cite{Or25} verbatim.
The notions themselves are not new:
aspherical classes and symplectic forms are classical (see, e.g., \cite{RT99}),
atoroidal ones appear in \cite{Ke09,Gu13,GM20},
and spherical monotonicity is the standard monotonicity condition of Floer theory;
toroidal (negative) monotonicity was introduced by Ginzburg and G\"{u}rel \cite{GG16}
in their study of non-contractible periodic orbits of Hamiltonian diffeomorphisms.

\begin{definition}[{\cite[Definition 2.1]{Or25}}]\label{def:atoroidal_aspherical}
    Let $u\in H^2(X;R)$ where $R=\RR$ or $\ZZ$.
    \begin{enumerate}
        \item $u$ is called \textit{aspherical} if $u$ vanishes on $h_S(\pi_2(X))$.
        \item $u$ is called \textit{atoroidal} if $u$ vanishes on $h_T(\pi_1(\mathcal{L}_{\alpha}X))$ for all $\alpha\in [S^1,X]$.
    \end{enumerate}
    Here ``$u$ vanishes on a subgroup $G\subset H_2(X;\ZZ)$'' means that
    the Kronecker pairing $\langle u,c\rangle\in R$ is zero for every $c\in G$.
\end{definition}

Let $(M,\omega)$ be a connected closed symplectic manifold and
$c_1=c_1(M,\omega)\in H^2(M;\ZZ)$ the first Chern class of the tangent bundle $TM$
with respect to an almost complex structure compatible with $\omega$.

\begin{definition}[{\cite[Definition 2.2]{Or25}}]\label{def:monotone}
    Let $(M,\omega)$ be a connected closed symplectic manifold.
    \begin{enumerate}
        \item $(M,\omega)$ is called \textit{toroidally monotone} (resp.\ \textit{toroidally negative monotone})
        if there exists a non-negative (resp.\ negative) number $\lambda\in\RR$ such that for all $\alpha\in [S^1,M]$,
        \[
    	[\omega]|_{h_T(\pi_1(\mathcal{L}_{\alpha}M))} = \lambda c_1|_{h_T(\pi_1(\mathcal{L}_{\alpha}M))}.
        \]
        \item $(M,\omega)$ is called \textit{spherically monotone} (resp.\ \textit{spherically negative monotone})
        if there exists a non-negative (resp.\ negative) number $\lambda\in\RR$ such that
        \[
    	[\omega]|_{h_S(\pi_2(M))} = \lambda c_1|_{h_S(\pi_2(M))}.
        \]
    \end{enumerate}
\end{definition}

For the definition of the Kodaira dimension $\kappa(M,\omega)$ of a closed symplectic $4$-manifold
we refer the reader to \cite[Section 2.2]{Or25} and \cite{Li06}.

Finally, we recall the cohomological lower bounds that we shall use in Section \ref{sec:ruled}.
For a field $\mathbb{F}$,
the \textit{cup-length} $\mathrm{cl}_{\mathbb{F}}(X)$ is the largest integer $m$ such that there exist classes
$u_1,\ldots,u_m\in H^{>0}(X;\mathbb{F})$ with $u_1\smile\cdots\smile u_m\neq 0$;
one has $\cat(X)\geq \mathrm{cl}_{\mathbb{F}}(X)+1$, see \cite[Proposition 1.5]{CLOT03}.
A class $\bar{u}\in H^{\ast}(X\times X;\mathbb{F})$ is a \textit{zero-divisor} if it restricts to zero on the diagonal;
for $u\in H^{\ast}(X;\mathbb{F})$ we write
\[
    \bar{u} := 1\times u - u\times 1 \in H^{\ast}(X\times X;\mathbb{F}),
\]
which is a zero-divisor.
The \textit{zero-divisor cup-length} $\mathrm{zcl}_{\mathbb{F}}(X)$ is the largest integer $m$
such that there exist zero-divisors $\bar{u}_1,\ldots,\bar{u}_m$ with $\bar{u}_1\smile\cdots\smile\bar{u}_m\neq 0$;
one has
\begin{equation}\label{eq:zcl}
    \TC(X)\geq \mathrm{zcl}_{\mathbb{F}}(X)+1,
\end{equation}
see \cite[Theorem 7]{Fa03}.


\section{From spherical to toroidal monotonicity}\label{sec:spherical_to_toroidal}

In this section we prove Theorem \ref{thm:A}.
We first record the elementary fact that the evaluation-based Definition \ref{def:atoroidal_aspherical}
agrees with the pullback-based definitions used in \cite{GM20,SS25}.

\begin{lemma}\label{lem:matching}
    Let $X$ be a connected CW-complex and $u\in H^2(X;R)$ with $R=\RR$ or $\ZZ$.
    \begin{enumerate}
        \item $u$ is aspherical if and only if $f^{\ast}u=0\in H^2(S^2;R)$ for every continuous map $f\colon S^2\to X$.
        \item $u$ is atoroidal if and only if $f^{\ast}u=0\in H^2(S^1\times S^1;R)$ for every continuous map $f\colon S^1\times S^1\to X$.
    \end{enumerate}
\end{lemma}

\begin{proof}
    (i) Since $H_1(S^2;\ZZ)=0$,
    the universal coefficient theorem identifies
    \[
        H^2(S^2;R)\cong\Hom(H_2(S^2;\ZZ),R),
    \]
    and under this identification $f^{\ast}u$ corresponds to the homomorphism $c\mapsto\langle u,f_{\ast}c\rangle$.
    Hence $f^{\ast}u=0$ if and only if $\langle u,f_{\ast}[S^2]\rangle=0$.
    Since every continuous map $f\colon S^2\to X$ represents,
    after a choice of a path to the basepoint, a class $[f]\in\pi_2(X)$ with $h_S([f])=f_{\ast}[S^2]$,
    and since the homology class $f_{\ast}[S^2]$ does not depend on this choice,
    we obtain
    \[
        \left\{\,f_{\ast}[S^2] \relmiddle| f\colon S^2\to X\,\right\} = h_S(\pi_2(X)),
    \]
    and the claim follows.

    (ii) Since $H_1(S^1\times S^1;\ZZ)\cong\ZZ^2$ is free,
    the universal coefficient theorem again gives
    \[
        H^2(S^1\times S^1;R)\cong\Hom(H_2(S^1\times S^1;\ZZ),R),
    \]
    so that $f^{\ast}u=0$ if and only if $\langle u,f_{\ast}[S^1\times S^1]\rangle=0$.
    By the exponential correspondence $C(S^1,C(S^1,X))\cong C(S^1\times S^1,X)$,
    a loop $c\colon S^1\to\mathcal{L}X$ based at $\gamma\in\mathcal{L}_{\alpha}X$ corresponds to
    a map $\widehat{c}\colon S^1\times S^1\to X$ with $\widehat{c}(0,\cdot)=\gamma$,
    and by definition $h_T([c])=\widehat{c}_{\ast}[S^1\times S^1]$.
    Conversely, every continuous map $f\colon S^1\times S^1\to X$ arises in this way from a loop in $\mathcal{L}_{\alpha}X$
    for $\alpha=[f(0,\cdot)]$.
    Hence
    \[
        \bigcup_{\alpha\in [S^1,X]} h_T\bigl(\pi_1(\mathcal{L}_{\alpha}X)\bigr)
        = \left\{\,f_{\ast}[S^1\times S^1] \relmiddle| f\colon S^1\times S^1\to X\,\right\},
    \]
    and, since vanishing on a set of generators is equivalent to vanishing on the subgroup they generate,
    the claim follows.
\end{proof}

The following lemma is the key step.
At the level of symplectic forms it is extracted from the proof of \cite[Theorem 4]{Bo16} for torsion-free fundamental groups;
the covering argument which removes the torsion-freeness assumption appears in \cite[Proposition 6.2]{SS25}.
We give a complete proof for the reader's convenience,
and because our evaluation-based formulation allows both $R=\ZZ$ and $R=\RR$ without any assumption on characteristics.

\begin{lemma}\label{lem:aspherical_implies_atoroidal}
    Let $X$ be a connected CW-complex whose fundamental group $\pi_1(X)$ contains no subgroup isomorphic to $\ZZ\oplus\ZZ$,
    and let $u\in H^2(X;R)$ with $R=\RR$ or $\ZZ$.
    If $u$ is aspherical, then $u$ is atoroidal.
\end{lemma}

\begin{proof}
    By Lemma \ref{lem:matching}, it suffices to show that
    $\langle u,w_{\ast}[T^2]\rangle=0$ for every continuous map $w\colon T^2=S^1\times S^1\to X$.
    Fix such a $w$, choose basepoints, and let
    \[
        \varphi := w_{\ast}\colon \ZZ^2=\pi_1(T^2)\to\pi:=\pi_1(X)
    \]
    be the induced homomorphism;
    a different choice of basepoints changes $\varphi$ by a conjugation, which does not affect the argument below.

    If $\varphi$ were injective, then $\Image\varphi\cong\ZZ^2$ would be a subgroup of $\pi$, contradicting our hypothesis.
    Hence $\Ker\varphi\neq 0$.
    Choose $0\neq v\in\Ker\varphi$ and write $v=d\,v_0$ with $v_0\in\ZZ^2$ primitive and $d\geq 1$.
    Then $\varphi(v_0)^d=\varphi(v)=1$, so $t:=\varphi(v_0)\in\pi$ has finite order;
    let $m\geq 1$ denote this order, so that $m$ divides $d$.

    \smallskip
    \textit{Step 1: reduction to a primitive kernel element.}
    Extend $v_0$ to a basis $(v_0,v_1)$ of $\ZZ^2$ and consider the index $m$ subgroup
        $\Lambda:=\ZZ(m v_0)\oplus\ZZ v_1\subset\ZZ^2$.
        Let $p\colon\widehat{T}\to T^2$ be the covering space corresponding to $\Lambda$.
        Then $\widehat{T}$ is again a torus, $p$ has degree $m$,
        and the basis element $\widehat{v}_0\in\pi_1(\widehat{T})\cong\ZZ^2$ mapping to $m v_0$
        is primitive in $\pi_1(\widehat{T})$ and satisfies
        \[
            (w\circ p)_{\ast}(\widehat{v}_0)=\varphi(m v_0)=t^m=1.
        \]
        Thus $g:=w\circ p\colon\widehat{T}\to X$ kills a primitive element of the fundamental group of its source torus.
    If $\pi$ is torsion-free, then $m=1$, $p=\mathrm{id}$, and this step is vacuous.

    \smallskip
    \textit{Step 2: a torus killing a primitive class factors through $S^2\vee S^1$.}
    Let $g\colon T^2\to X$ be any continuous map such that $g_{\ast}(e_1)=1\in\pi$ for some primitive element $e_1\in\pi_1(T^2)$.
        Choose a product decomposition $T^2=S^1\times S^1$ whose first factor represents $e_1$.
        The loop $g|_{S^1\times\{s_0\}}$ is trivial in $\pi_1(X)$ and hence null-homotopic.
        Since $S^1\times\{s_0\}\subset T^2$ is a cofibration,
        $g$ is homotopic to a map $g'$ which is constant on $S^1\times\{s_0\}$,
        and $g'$ therefore factors as
        \[
            g'=\bar{g}\circ q,\qquad q\colon T^2\to T^2/\bigl(S^1\times\{s_0\}\bigr).
        \]
        In the standard CW structure of $T^2$
        (one $0$-cell, two $1$-cells $a,b$ with $a$ the collapsed circle, one $2$-cell attached along $aba^{-1}b^{-1}$),
        the quotient has one $0$-cell, the $1$-cell $b$,
        and a $2$-cell attached along the null-homotopic word $bb^{-1}$;
        hence
        \[
            T^2/\bigl(S^1\times\{s_0\}\bigr)\simeq S^2\vee S^1,
        \]
        and on cellular chains $q$ maps the $2$-cell of $T^2$ homeomorphically onto the $2$-cell of the quotient,
        so that $q_{\ast}[T^2]=\pm[S^2]\in H_2(S^2\vee S^1;\ZZ)$.
        Writing $\iota\colon S^2\hookrightarrow S^2\vee S^1$ for the inclusion, we conclude
        \[
            g_{\ast}[T^2]=g'_{\ast}[T^2]=\pm(\bar{g}\circ\iota)_{\ast}[S^2]\in h_S(\pi_2(X)).
        \]

    \smallskip
    \textit{Step 3: conclusion.}
    Applying Step 2 to $g=w\circ p$ and using the asphericity of $u$,
        we obtain $\langle u,(w\circ p)_{\ast}[\widehat{T}\,]\rangle=0$.
        On the other hand, since $p$ has degree $m$,
        \[
            (w\circ p)_{\ast}[\widehat{T}\,]=w_{\ast}\bigl(p_{\ast}[\widehat{T}\,]\bigr)=m\cdot w_{\ast}[T^2],
        \]
        whence $m\,\langle u,w_{\ast}[T^2]\rangle=0$ in $R$.
    As $R=\RR$ or $\ZZ$ is torsion-free as an abelian group, $\langle u,w_{\ast}[T^2]\rangle=0$.
    This completes the proof.
\end{proof}

\begin{remark}
    Related results in degree two,
    including the chain of inclusions
    \[
        \{\text{bounded}\}\subset\{\text{hyperbolic}\}\subset\{\text{atoroidal}\}\subset\{\text{aspherical}\}
    \]
    and the behaviour of atoroidal classes under connected sums,
    can be found in \cite{Ke09,BKS24}.
\end{remark}

\begin{proposition}\label{prop:transfer}
    Let $(M,\omega)$ be a connected closed symplectic manifold which is spherically monotone
    $($resp.\ spherically negative monotone$)$ with monotonicity constant $\lambda$.
    If $\pi_1(M)$ contains no subgroup isomorphic to $\ZZ\oplus\ZZ$,
    then $(M,\omega)$ is toroidally monotone $($resp.\ toroidally negative monotone$)$
    with the same monotonicity constant $\lambda$.
\end{proposition}

\begin{proof}
    Set $u:=[\omega]-\lambda c_1\in H^2(M;\RR)$.
    Spherical (negative) monotonicity with constant $\lambda$ says precisely that $u$ vanishes on $h_S(\pi_2(M))$,
    i.e., that $u$ is aspherical.
    By Lemma \ref{lem:aspherical_implies_atoroidal}, $u$ is atoroidal,
    i.e., $u$ vanishes on $h_T(\pi_1(\mathcal{L}_{\alpha}M))$ for every $\alpha\in [S^1,M]$.
    This is exactly the statement that
    \[
        [\omega]|_{h_T(\pi_1(\mathcal{L}_{\alpha}M))}=\lambda c_1|_{h_T(\pi_1(\mathcal{L}_{\alpha}M))}
        \qquad\text{for all }\alpha\in [S^1,M],
    \]
    i.e., that $(M,\omega)$ is toroidally monotone (if $\lambda\geq 0$)
    resp.\ toroidally negative monotone (if $\lambda<0$) with constant $\lambda$.
\end{proof}

\begin{proof}[Proof of Theorem \ref{thm:A}]
    By Proposition \ref{prop:transfer},
    $(M,\omega)$ is toroidally monotone or toroidally negative monotone with constant $\lambda$.
    Together with the hypothesis $([\omega]-\lambda c_1)^n\neq 0$,
    \cite[Theorem 4.3]{Or25} yields $\TC(M)=4n+1$.
    The equality $\cat(M)=2n+1$ follows from \cite[Theorem 4.6]{Or25},
    which requires only spherical (negative) monotonicity and the non-vanishing of $([\omega]-\lambda c_1)^n$.
\end{proof}

\begin{proof}[Proof of Corollary \ref{cor:A_dim4}]
    Let $\lambda\geq 0$ be the monotonicity constant.
    By Theorem \ref{thm:A}, it suffices to show that $([\omega]-\lambda c_1)^2\neq 0$;
    note that hyperbolic groups contain no $\ZZ\oplus\ZZ$ since all their abelian subgroups are virtually cyclic,
    see \cite[Corollary III.$\Gamma$.3.10]{BH99}.

    By symplectically blowing down a maximal disjoint collection of exceptional spheres \cite{Mc90,Mc91}
    (see also \cite[Section 2.2]{Or25}),
    the manifold $M$ is diffeomorphic to $M_{\min}\mathbin{\#}k\barCP$
    for some $k\geq 0$ and a minimal symplectic $4$-manifold $(M_{\min},\omega_{\min})$,
    and $\kappa(M,\omega)=\kappa(M_{\min},\omega_{\min})$ by the definition of the Kodaira dimension.
    Let $c\colon M\to M_{\min}$ be the degree one map collapsing the summands $\barCP$,
    let $E_1,\ldots,E_k\in H_2(M;\ZZ)$ be the exceptional classes,
    and let $e_i\in H^2(M;\ZZ)$ be the Poincar\'{e} dual of $E_i$,
    so that $\langle e_i,E_j\rangle=E_i\cdot E_j=-\delta_{ij}$.
    Then
    \[
        [\omega]=c^{\ast}[\omega_{\min}]-\sum_{i=1}^{k}\varepsilon_i e_i,
        \qquad
        c_1(M,\omega)=c^{\ast}c_1(M_{\min},\omega_{\min})-\sum_{i=1}^{k}e_i,
    \]
    where $\varepsilon_i=\langle[\omega],E_i\rangle>0$ is the symplectic area of the $i$-th exceptional sphere;
    see \cite[equations (7.1.30) and (7.1.31)]{MS17}, where $\varepsilon_i=\pi\lambda_i^2$ in the notation there.
    Since exceptional spheres are embedded spheres, their classes lie in $h_S(\pi_2(M))$,
    and spherical monotonicity yields $\varepsilon_i=\lambda\langle c_1,E_i\rangle=\lambda$ for every $i$.
    Therefore
    \begin{align*}    
        [\omega]-\lambda c_1
        &=c^{\ast}\bigl([\omega_{\min}]-\lambda c_1(M_{\min},\omega_{\min})\bigr)
        +\sum_{i=1}^{k}(\lambda-\varepsilon_i)e_i\\
        &=c^{\ast}\bigl([\omega_{\min}]-\lambda c_1(M_{\min},\omega_{\min})\bigr).
    \end{align*}
    Since $c$ has degree one, writing $K_{\min}=-c_1(M_{\min},\omega_{\min})$ we obtain
    \[
        ([\omega]-\lambda c_1)\cdot([\omega]-\lambda c_1)
        =\int_{M_{\min}}\omega_{\min}\wedge\omega_{\min}
        +2\lambda\,K_{\min}\cdot[\omega_{\min}]+\lambda^2\,K_{\min}\cdot K_{\min}.
    \]
    As $(M_{\min},\omega_{\min})$ is minimal and $\kappa(M_{\min},\omega_{\min})=\kappa(M,\omega)\neq-\infty$,
    we have $K_{\min}\cdot K_{\min}\geq 0$ and $K_{\min}\cdot[\omega_{\min}]\geq 0$
    by the definition of the Kodaira dimension (see \cite[Section 2.2]{Or25} and \cite{Li06}).
    Hence the right-hand side is positive, and in particular $([\omega]-\lambda c_1)^2\neq 0$.
    Now Theorem \ref{thm:A} applies with $n=2$.
\end{proof}

\begin{remark}\label{rem:minimal_model}
    The proofs of \cite[Theorems 1.4 and 1.9]{Or25} invoke the inequalities
    $K\cdot K\geq 0$ and $K\cdot[\omega]\geq 0$ for the canonical class $K$ of $(M,\omega)$ itself.
    These hold by the definition of the Kodaira dimension when $(M,\omega)$ is \textit{minimal},
    but $K\cdot K$ may become negative after blowing up,
    since each blowup decreases $K\cdot K$ by one
    (consider, e.g., the blowup of a $K3$ surface, for which $K\cdot K=-1$ while $\kappa=0$).
    The argument given in the proof of Corollary \ref{cor:A_dim4} above,
    which reduces to the minimal model by means of the constraint
    $\langle[\omega],E_i\rangle=\lambda\langle c_1,E_i\rangle$
    imposed by spherical monotonicity on the exceptional classes,
    completes the proofs of \cite[Theorems 1.4 and 1.9]{Or25} in the non-minimal case;
    the statements of loc.\ cit.\ are unaffected.
\end{remark}

\begin{proof}[Proof of Corollary \ref{cor:group_theoretic}]
    Suppose that $\kappa(M,\omega)=-\infty$.
    Then, by the classification recalled in Section \ref{sec:introduction}
    (see \cite{LL95,Liu96,OO96a,OO96b} and \cite[Remark 1.6]{Or25}),
    $M$ is a blowup of a rational or of a ruled symplectic $4$-manifold,
    so that $\pi_1(M)$ is either trivial or isomorphic to $\pi_1(\Sigma_g)$ for some $g\geq 1$.
    Now, $M$ is not simply connected by our standing assumption,
    $\pi_1(M)\not\cong\pi_1(\Sigma_1)=\ZZ\oplus\ZZ$ since $\pi_1(M)$ contains no $\ZZ\oplus\ZZ$,
    and $\pi_1(M)\not\cong\pi_1(\Sigma_g)$ for $g\geq 2$ by hypothesis.
    Therefore $\kappa(M,\omega)\neq-\infty$, and Corollary \ref{cor:A_dim4} applies.
\end{proof}

\begin{remark}[Comparison with Theorem \ref{thm:previous} (ii)]\label{rem:comparison}
    Corollary \ref{cor:A_dim4} strengthens Theorem \ref{thm:previous} (ii) in three ways,
    at the cost of replacing the hypotheses there by the single condition that $\pi_1(M)$ contain no $\ZZ\oplus\ZZ$.
    \begin{enumerate}
        \item The conclusion is the equality $\TC(M)=9$ rather than the dichotomy $\TC(M)\in\{8,9\}$.
        \item The hypothesis is strictly weaker:
        if the centralizer of every non-trivial element of $\pi_1(M)$ is infinite cyclic,
        then $\pi_1(M)$ contains no $\ZZ\oplus\ZZ$
        (the centralizer of a generator of such a subgroup would contain the subgroup itself),
        while the converse fails:
        any group with torsion violates the centralizer condition
        (the centralizer of a non-trivial torsion element contains that element and hence is not infinite cyclic),
        so, e.g., hyperbolic groups with torsion qualify,
        as do the torsion-free solvable Baumslag--Solitar groups $BS(1,m)$ with $m\geq 2$.
        Moreover, no finiteness assumption (type $FL$) is needed, and torsion in $\pi_1(M)$ is allowed.
        \item The proof does not pass through the Eilenberg--MacLane space $K(\pi_1(M),1)$.
    \end{enumerate}
    The mechanism behind (i) can be traced in the obstruction-theoretic proof of \cite[Theorem 3]{FM20}:
    in the spectral sequence computing essential cohomology classes,
    the hypothesis on centralizers kills all obstructions except the final one,
    which lives in the row governed by $H^1$ of the (infinite cyclic) joint centralizers ---
    algebraically, the ``toroidal direction''.
    Passing from spherical to toroidal monotonicity via Proposition \ref{prop:transfer}
    eliminates precisely this obstruction geometrically.
\end{remark}

\begin{remark}[On the weight estimate underlying {\cite[Theorem 4.3]{Or25}}]\label{rem:SS_repair}
    The proof of \cite[Theorem 4.3]{Or25} rests on the estimate
    $\wgt\bigl([\pr_2^{\ast}\eta-\pr_1^{\ast}\eta]\bigr)\geq 2$ for atoroidal closed $2$-forms $\eta$,
    due to Grant and Mescher \cite[Corollary 3.6]{GM20}.
    Sandrock and Schick \cite[Remark 3.2]{SS25} observed a gap in the Mayer--Vietoris argument
    in the proof of \cite[Theorem 3.5]{GM20}
    (a snake-lemma computation performed on a sequence of cochain complexes
    which is only chain homotopy equivalent to an exact one)
    and provided a corrected and generalized proof,
    valid for arbitrary CW-complexes, singular cohomology and suitable field coefficients,
    which moreover yields the equality $\wgt=2$.
    All results of \cite{Or25} and of the present paper are therefore unaffected;
    we cite \cite{GM20} together with \cite{SS25} throughout.
\end{remark}

\begin{remark}\label{rem:Kodaira_Thurston}
    The hypothesis on $\pi_1(M)$ delimits the method:
    the Kodaira--Thurston manifold is spherically monotone with $\kappa=0$,
    but its fundamental group contains $\ZZ\oplus\ZZ$, so Theorem \ref{thm:A} does not apply.
    In view of the upper bounds for the topological complexity of spaces with nilpotent fundamental group \cite{Gr12},
    it is natural to expect $\TC$ to be non-maximal in this case;
    determining it would test the sharpness of the hypothesis (i) of Theorem \ref{thm:A}.
\end{remark}


\section{Ruled surfaces and their blowups}\label{sec:ruled}

In this section we prove Theorem \ref{thm:B}.
Throughout, $\Sigma_g$ denotes a closed orientable surface of genus $g\geq 1$
and $E\to\Sigma_g$ an $S^2$-bundle with structure group $SO(3)$;
up to isomorphism there are exactly two such bundles,
the trivial one $\Sigma_g\times S^2$ and the twisted one,
distinguished by the second Stiefel--Whitney class of the associated rank three real vector bundle.
For $k\geq 0$ we write $M_k:=E\mathbin{\#}k\barCP$, the $k$-fold blowup.
All cohomology groups in this section are taken with rational coefficients unless stated otherwise.

We first record the well known structure of the rational cohomology ring of $E$.

\begin{lemma}\label{lem:ring}
    Let $E\to\Sigma_g$ be an $S^2$-bundle as above.
    Then there is an isomorphism of graded rings
    \[
        H^{\ast}(E;\QQ)\cong H^{\ast}(\Sigma_g;\QQ)\otimes H^{\ast}(S^2;\QQ).
    \]
    More precisely,
    let $a_1,b_1,\ldots,a_g,b_g\in H^1(E;\QQ)$ and $F\in H^2(E;\QQ)$ denote the pullbacks of
    a standard symplectic basis of $H^1(\Sigma_g;\QQ)$ and of the fundamental cohomology class of $\Sigma_g$, respectively.
    Then there exists a class $y\in H^2(E;\QQ)$ with
    \[
        y^2=0,\qquad
        \langle yF,[E]\rangle\neq 0,
    \]
    and $H^{\ast}(E;\QQ)$ is generated by $a_1,b_1,\ldots,a_g,b_g$ and $y$ subject to the relations pulled back from $\Sigma_g$,
    namely $a_ib_j=\delta_{ij}F$, $a_ia_j=b_ib_j=0$, $F a_i=F b_i=0$ and $F^2=0$.
\end{lemma}

\begin{proof}
    Since the structure group $SO(3)$ is connected,
    the fundamental group of the base acts trivially on $H^{\ast}(S^2;\QQ)$,
    and the Serre spectral sequence of $S^2\to E\to\Sigma_g$ has
    $E_2^{p,q}=H^p(\Sigma_g;\QQ)\otimes H^q(S^2;\QQ)$.
    All differentials vanish for degree reasons
    ($H^{\mathrm{odd}}(S^2;\QQ)=0$ kills $d_2$, and $d_3$ lands in $E_3^{p+3,q-2}=0$ since $\dim\Sigma_g=2$),
    so the spectral sequence collapses and, in particular,
    the restriction map $H^2(E;\QQ)\to H^2(S^2;\QQ)$ to a fibre is surjective.
    Choose $\tilde{y}\in H^2(E;\QQ)$ restricting to a generator of $H^2(S^2;\QQ)$.
    Integration along the fibre shows $\langle\tilde{y}F,[E]\rangle=\pm 1$;
    in particular, $H^4(E;\QQ)\cong\QQ$ is spanned by $\tilde{y}F$,
    and we may write $\tilde{y}^2=c\,\tilde{y}F$ for some $c\in\QQ$.
    The class $y:=\tilde{y}-\tfrac{c}{2}F$ then satisfies
    $y^2=\tilde{y}^2-c\,\tilde{y}F+\tfrac{c^2}{4}F^2=0$ and $yF=\tilde{y}F\neq 0$.
    The relations among the pulled back classes hold already in $H^{\ast}(\Sigma_g;\QQ)$,
    and comparing dimensions via the collapsed spectral sequence shows that the resulting ring map
    $H^{\ast}(\Sigma_g;\QQ)\otimes H^{\ast}(S^2;\QQ)\to H^{\ast}(E;\QQ)$ is an isomorphism.
\end{proof}

Next we compute the Lusternik--Schnirelmann category of $M_k$.
Our tool is the Berstein--Schwarz class.
Let $X$ be a connected CW-complex with fundamental group $\pi$,
let $I=\Ker(\varepsilon\colon\ZZ[\pi]\to\ZZ)$ be the augmentation ideal,
and let $b_X\in H^1(X;I)$ denote the \textit{Berstein--Schwarz class} of $X$,
that is, the pullback of the class $b_{\pi}\in H^1(\pi;I)$
associated with the extension $0\to I\to\ZZ[\pi]\to\ZZ\to 0$
under a classifying map $X\to K(\pi,1)$;
we refer the reader to \cite{DR09} and \cite[Section 3]{FM20} for details.
The Berstein--\v{S}varc theorem
(\cite[Theorem 20]{Sch66} and \cite[Theorem A]{Ber76} for $n\geq 3$; the remaining case $n=2$ is due to \cite{DR09})
states that a connected CW-complex $X$ with $\dim X=n$ satisfies
$\cat(X)=n+1$ if and only if $b_X^n\neq 0\in H^n(X;I^{\otimes n})$.
In particular, if $\mathrm{cd}(\pi)<n=\dim X$,
then $b_X^n$ is pulled back from $H^n(\pi;I^{\otimes n})=0$ and hence vanishes,
so that
\begin{equation}\label{eq:BS_bound}
    \mathrm{cd}(\pi_1(X))<\dim X
    \quad\Longrightarrow\quad
    \cat(X)\leq\dim X.
\end{equation}

\begin{lemma}\label{lem:cat}
    Let $g\geq 1$ and $k\geq 0$. Then $\cat(M_k)=4$.
\end{lemma}

\begin{proof}
    We first treat the case $k=0$, i.e., $M_0=E$.
    By Lemma \ref{lem:ring} we have $a_1b_1y=Fy\neq 0$,
    so $\mathrm{cl}_{\QQ}(E)\geq 3$ and hence $\cat(E)\geq 4$.
    On the other hand, since $S^2$ is simply connected,
    the bundle projection $E\to\Sigma_g$ induces an isomorphism of fundamental groups,
    so $\pi_1(E)\cong\pi_1(\Sigma_g)$ and $\mathrm{cd}(\pi_1(E))=2$,
    as $\Sigma_g$ is a closed aspherical surface for $g\geq 1$.
    Since $\mathrm{cd}(\pi_1(E))=2<4=\dim E$,
    the implication \eqref{eq:BS_bound} yields $\cat(E)\leq 4$.

    For $k\geq 1$, let $c\colon M_k\to E$ be the map collapsing the connected summands $\barCP$ to points.
    Then $c$ has degree one, and hence $c^{\ast}$ is injective on rational cohomology:
    indeed, by the projection formula,
    $c_{\ast}\bigl(c^{\ast}(x)\frown[M_k]\bigr)=x\frown c_{\ast}[M_k]=x\frown[E]$ for every $x\in H^{\ast}(E;\QQ)$,
    and $\frown[E]$ is an isomorphism by Poincar\'e duality.
    Therefore $c^{\ast}(a_1)c^{\ast}(b_1)c^{\ast}(y)=c^{\ast}(Fy)\neq 0$,
    which gives $\cat(M_k)\geq 4$ as before.
    Moreover, $c$ induces an isomorphism of fundamental groups,
    so $\pi_1(M_k)\cong\pi_1(E)\cong\pi_1(\Sigma_g)$ and $\mathrm{cd}(\pi_1(M_k))=2<4=\dim M_k$.
    Hence $\cat(M_k)\leq 4$ by \eqref{eq:BS_bound}.
\end{proof}

We are now in a position to prove Theorem \ref{thm:B}.

\begin{proof}[Proof of Theorem \ref{thm:B}]
    Let $g\geq 2$ and $k\geq 0$, and let $M=M_k$ be as above.
    The equality $\cat(M)=4$ is Lemma \ref{lem:cat}.

    \textit{Upper bound for $\TC$.}
    By \cite[Theorem 5]{Fa03} and Lemma \ref{lem:cat},
    \[
        \TC(M)\leq 2\cat(M)-1=7.
    \]

    \textit{Lower bound for $\TC$.}
    We first treat the case $k=0$.
    In the notation of Lemma \ref{lem:ring},
    consider the zero-divisors $\bar{a}_1,\bar{b}_1,\bar{a}_2,\bar{b}_2,\bar{y}\in H^{\ast}(E\times E;\QQ)$,
    where $\bar{u}=1\times u-u\times 1$.
    A direct computation using $F=a_1b_1=a_2b_2$ gives
    \[
        \bar{a}_1\bar{b}_1
        = 1\times F + F\times 1 + b_1\times a_1 - a_1\times b_1,
    \]
    and similarly for $\bar{a}_2\bar{b}_2$.
    Multiplying these two expressions,
    all mixed terms vanish
    because products of three or more one-dimensional classes are pulled back from $H^{\geq 3}(\Sigma_g;\QQ)=0$,
    because $a_ib_j=a_ia_j=b_ib_j=0$ for $i\neq j$,
    and because $F^2=Fa_i=Fb_i=0$;
    hence
    \[
        \bar{a}_1\bar{b}_1\bar{a}_2\bar{b}_2 = 2\,F\times F.
    \]
    Since $y^2=0$, we also have $\bar{y}^2=-2\,y\times y$, and therefore
    \begin{align*}
        \bar{a}_1\bar{b}_1\bar{a}_2\bar{b}_2\,\bar{y}^2
        &= (2\,F\times F)(-2\,y\times y)
        = -4\,(Fy)\times(Fy)\neq 0\\
        &\in H^4(E;\QQ)\otimes H^4(E;\QQ)\subset H^8(E\times E;\QQ).
    \end{align*}
    Thus $\mathrm{zcl}_{\QQ}(E)\geq 6$ and $\TC(E)\geq 7$ by \eqref{eq:zcl}.

    For $k\geq 1$, let $c\colon M_k\to E$ be the degree one collapse map as in the proof of Lemma \ref{lem:cat}.
    Then $c\times c$ is a degree one map between closed oriented $8$-manifolds,
    so $(c\times c)^{\ast}$ is injective on rational cohomology.
    Since $(c\times c)^{\ast}(\bar{u})=\overline{c^{\ast}u}$ for every $u\in H^{\ast}(E;\QQ)$,
    the product of the pulled back zero-divisors
    \[
        \overline{c^{\ast}a_1}\;\overline{c^{\ast}b_1}\;\overline{c^{\ast}a_2}\;\overline{c^{\ast}b_2}\;\bigl(\overline{c^{\ast}y}\bigr)^2
        =(c\times c)^{\ast}\bigl(\bar{a}_1\bar{b}_1\bar{a}_2\bar{b}_2\,\bar{y}^2\bigr)
    \]
    is non-zero.
    Hence $\mathrm{zcl}_{\QQ}(M_k)\geq 6$ and $\TC(M_k)\geq 7$.
    This completes the proof.
\end{proof}

For completeness, we record the simply connected case,
which follows from \cite{FTY03} and the dimension--connectivity upper bounds.

\begin{theorem}\label{thm:simply_connected_ruled}
    Let $(M,\omega)$ be a simply connected closed symplectic $4$-manifold.
    Then $\cat(M)=3$ and $\TC(M)=5$.
\end{theorem}

\begin{proof}
    Since $M$ is simply connected,
    we have $\cat(M)\leq\frac{\dim M}{2}+1=3$ and $\TC(M)\leq\dim M+1=5$,
    see \cite[Theorem 1.50]{CLOT03} and \cite[Theorem 5.2]{Fa04}.
    Since $[\omega]^2\neq 0$, we have $\mathrm{cl}_{\RR}(M)\geq 2$ and hence $\cat(M)\geq 3$.
    The equality $\TC(M)=5$ is \cite[Corollary 3.2]{FTY03}.
\end{proof}

\begin{proof}[Proof of Corollary \ref{cor:trichotomy}]
    (i) is Theorem \ref{thm:simply_connected_ruled}.
    For (ii), by the classification cited in the proof of Corollary \ref{cor:group_theoretic},
    the non-simply connected closed symplectic $4$-manifolds with $\kappa=-\infty$ and base genus $g\geq 2$
    are exactly the manifolds $M_k$ of Theorem \ref{thm:B}.
    (iii) is Corollary \ref{cor:A_dim4}.
\end{proof}

In the genus one case, our method yields only the following partial result.

\begin{proposition}\label{prop:genus_one}
    Let $M$ be a blowup of an $S^2$-bundle over $T^2=\Sigma_1$.
    Then $\cat(M)=4$ and $5\leq\TC(M)\leq 7$.
    Moreover, $\TC(T^2\times S^2)=5$.
\end{proposition}

\begin{proof}
    The equality $\cat(M)=4$ and the upper bound $\TC(M)\leq 7$ are proved exactly as in
    Lemma \ref{lem:cat} and Theorem \ref{thm:B}.
    For the lower bound, in the notation of Lemma \ref{lem:ring} (with $g=1$) we compute
    \begin{align*}
        \bar{a}_1\bar{b}_1\,\bar{y}^2
        &=\bigl(1\times F+F\times 1+b_1\times a_1-a_1\times b_1\bigr)\bigl(-2\,y\times y\bigr)\\
        &=-2\,y\times Fy-2\,Fy\times y-2\,b_1y\times a_1y+2\,a_1y\times b_1y,
    \end{align*}
    whose K\"{u}nneth component in $H^2(E;\QQ)\otimes H^4(E;\QQ)$ equals $-2\,y\times Fy\neq 0$.
    Hence $\mathrm{zcl}_{\QQ}(E)\geq 4$ and $\TC(E)\geq 5$;
    for blowups one pulls back along the degree one collapse map as before.
    Finally, for the trivial bundle,
    the product inequality $\TC(X\times Y)\leq\TC(X)+\TC(Y)-1$ \cite[Theorem 11]{Fa03}
    together with $\TC(S^2)=3$ \cite[Theorem 8]{Fa03} and $\TC(T^2)=3$ \cite[Theorem 13]{Fa03} gives
    $\TC(T^2\times S^2)\leq 5$, whence $\TC(T^2\times S^2)=5$.
\end{proof}

\begin{problem}\label{prob:genus_one}
    Determine $\TC(M)\in\{5,6,7\}$ for the twisted $S^2$-bundle over $T^2$
    and for the blowups of $S^2$-bundles over $T^2$.
\end{problem}

We conclude this section by explaining how Theorem \ref{thm:B} shows the sharpness of Theorem \ref{thm:A}.

\begin{remark}[Sharpness of the hypotheses of Theorem \ref{thm:A}]\label{rem:sharpness}
    Let $g\geq 2$ and equip $M=\Sigma_g\times S^2$ with a product symplectic form
    $\omega=\omega_{\Sigma}\oplus\omega_{S^2}$,
    where $\omega_{S^2}$ is normalized so that $[\omega_{S^2}]=\lambda\,c_1(S^2)$ in $H^2(S^2;\RR)$ for some $\lambda>0$.
    By \cite[Remark 1.7]{Or25}
    (applied to the symplectically atoroidal factor $\Sigma_g$ and the strongly monotone factor $S^2=\CP^1$),
    the manifold $(M,\omega)$ is toroidally monotone with monotonicity constant $\lambda$,
    and $\kappa(M,\omega)=-\infty$ since $M$ is ruled.
    Moreover, $\pi_1(M)=\pi_1(\Sigma_g)$ is hyperbolic, hence contains no $\ZZ\oplus\ZZ$.
    Nevertheless, $\TC(M)=7\neq 9$ and $\cat(M)=4\neq 5$ by Theorem \ref{thm:B}.

    This is consistent with Theorem \ref{thm:A}:
    the class $u=[\omega]-\lambda c_1$ decomposes as
    $u=\bigl([\omega_{\Sigma}]-\lambda c_1(\Sigma_g)\bigr)\oplus\bigl([\omega_{S^2}]-\lambda c_1(S^2)\bigr)$,
    and the second summand vanishes by our normalization,
    so that $u$ is pulled back from $\Sigma_g$ and $u^2=0$.
    Thus the hypothesis (ii) of Theorem \ref{thm:A} fails,
    and the example shows that it cannot be removed
    even in the presence of toroidal monotonicity and a hyperbolic fundamental group.
    Equivalently, in Corollary \ref{cor:A_dim4} the assumption $\kappa(M,\omega)\neq-\infty$ is indispensable.
\end{remark}


\subsection*{Use of AI}
The author made substantial use of the large language model Claude (Anthropic; model Claude Fable 5) during the development of this work.
In particular, the key ideas of the proofs of Theorems \ref{thm:A} and \ref{thm:B}
--- notably the passage from spherical to toroidal monotonicity via Lemma \ref{lem:aspherical_implies_atoroidal} ---
as well as initial versions of the proofs and of the text of the manuscript were obtained in interaction with the model,
which was also used for literature searches and for checking references.
The author has independently verified all arguments, statements, and references against the original sources,
and takes full responsibility for all claims, proofs, and text in this article.

\bibliographystyle{amsalpha}
\bibliography{orita_bibtex}
\end{document}